\documentclass[a4paper,11pt,reqno]{amsart}
\usepackage{amsmath}
\usepackage{amsthm,enumerate}

\usepackage{graphicx}
\usepackage{amssymb}

\usepackage{appendix}

\usepackage[english]{babel}

\usepackage[utf8x]{inputenc}
\usepackage{fancyhdr}

\usepackage{calc}
\usepackage{url}

\usepackage[text={6in,8.6in},centering]{geometry}

\usepackage{srcltx}

\usepackage[T1]{fontenc}

\usepackage[usenames,dvipsnames]{color}

\theoremstyle{plain}
\newtheorem{thm}{Theorem}[section]
\theoremstyle{plain}
\newtheorem{lem}[thm]{Lemma}
\newtheorem{prop}[thm]{Proposition}
\newtheorem{cor}[thm]{Corollary}

\theoremstyle{definition}

\newtheorem{rem}{Remark}[section]

\DeclareMathOperator{\Div}{div}

\newcommand{\hn}{\mathbb{H}^{N}}

\newcommand{\authorfootnotes}{\renewcommand\thefootnote{\@fnsymbol\c@footnote}}%

\numberwithin{equation}{section} \allowdisplaybreaks

\begin{document}
        \title[]{Improved L$^p$-Poincar\'e inequalities\\ on the hyperbolic space}

\date{}

\author[Elvise BERCHIO]{Elvise BERCHIO}
\address{\hbox{\parbox{5.7in}{\medskip\noindent{Dipartimento di Scienze Matematiche, \\
Politecnico di Torino,\\
        Corso Duca degli Abruzzi 24, 10129 Torino, Italy. \\[3pt]
        \em{E-mail address: }{\tt elvise.berchio@polito.it}}}}}
        \author[Lorenzo D'Ambrosio]{Lorenzo D'AMBROSIO}
\address{\hbox{\parbox{5.7in}{\medskip\noindent{Dipartimento di Matematica, \\
Universita' degli Studi di Bari ,\\
        via E. Orabona 4, I-70125 Bari, Italy. \\[3pt]
        \em{E-mail address: }{\tt lorenzo.dambrosio@uniba.it}}}}}
\author[Debdip GANGULY]{Debdip GANGULY}
\address{\hbox{\parbox{5.7in}{\medskip\noindent{Department of Mathematics,\\
 Technion, Israel Institute of Technology,\\
        Haifa 32000, Israel. \\[3pt]
        \em{E-mail address: }{\tt gdebdip@technion.ac.il}}}}}
\author[Gabriele GRILLO]{Gabriele GRILLO}
\address{\hbox{\parbox{5.7in}{\medskip\noindent{Dipartimento di Matematica,\\
Politecnico di Milano,\\
   Piazza Leonardo da Vinci 32, 20133 Milano, Italy. \\[3pt]
        \em{E-mail addresses: }{\tt
          gabriele.grillo@polimi.it}}}}}




\begin{abstract}
We investigate the possibility of improving the $p$-Poincar\'e inequality $\|\nabla_{\hn} u\|^p_p \\ \ge \Lambda_p \|u\|^p_p$ on the hyperbolic space, where $p>1$ and $\Lambda_p:=[(N-1)/p]^{p}$ is the best constant for which such inequality holds. We prove several different, and independent, improved inequalities, one of which is a Poincar\'e-Hardy inequality, namely an improvement of the best $p$-Poincar\'e inequality in terms of the Hardy weight $r^{-p}$, $r$ being geodesic distance from a given pole. Certain Hardy-Maz'ya-type inequalities in the Euclidean half-space are also obtained.
  \end{abstract}

\maketitle

 \section{Introduction}

Let $\hn$ denote the hyperbolic space of dimension $N\ge2$, $\nabla_{\hn},\Delta_{\hn}$ and ${\rm d}v_{\hn}$ its Riemannian gradient, Laplacian and measure, respectively. It is well known that the L$^2$ spectrum of $-\Delta_{\hn}$ is bounded away from zero. More precisely one has $\sigma(-\Delta_{\hn})=[(N-1)^2/4,+\infty)$. As a byproduct, the quadratic form inequality
\[
\int_{{\mathbb H}^N}|\nabla_{\hn} u|^2\,{\rm d}v_{\hn}\ge \frac{(N-1)^2}4\int_{{\mathbb H}^N}u^2\,{\rm d}v_{\hn}
\]
holds for all $u\in C_c^\infty(\hn)$. See e.g. \cite{DA} for an elementary proof. Besides, another inequality which one is very familiar within the Euclidean setting, namely \it Hardy's inequality\rm, holds true as well on $\hn$, so that one has, at least for $N\ge3$,
\[
\int_{{\mathbb H}^N}|\nabla_{\hn} u|^2\,{\rm d}v_{\hn}\ge \frac{(N-2)^2}4\int_{{\mathbb H}^N}\frac{u^2}{r^2}\,{\rm d}v_{\hn},
\]
where $r:=\varrho(x,x_0)$ denotes geodesic distance from a fixed pole $x_0$. In fact, such inequality
holds on any Cartan-Hadamard manifold, where the latter are defined as those manifolds
which are complete, simply connected and have nonpositive sectional curvatures. See \cite{Carron} for details. Hardy-type inequalities have been the object of a large amount of research in the past decades, see for example, with no claim of completeness, \cite{BFT2,BFT,Mitidieri1,Mitidieri2,BrezisM,Brezis,Dambrosio,pinch,pinch2,FT,GM,Kombe1,Kombe2,MMP,Mitidieri,TZ,Yang}.

\medskip

A combination of these inequalities was given in \cite{AK} and then rediscovered by other methods in \cite{BGG}. A simplified version of it reads
\begin{equation}\label{L2}
\int_{{\mathbb H}^N}|\nabla_{\hn} u|^2\,{\rm d}v_{\hn}- \frac{(N-1)^2}4\int_{{\mathbb H}^N}u^2\,{\rm d}v_{\hn}\ge \frac{1}4\int_{{\mathbb H}^N}\frac{u^2}{r^2}\,{\rm d}v_{\hn}
\end{equation}
for all $u\in C_c^\infty(\hn)$, and  the constants in \eqref{L2} are sharp (the sharpness of the constant $(N-1)^2/4$ in the l.h.s. being obvious), see \cite{BGG}. The sharpness of related inequalities in more general manifolds and similar  improved inequalities of Rellich type, which are again sharp in suitable senses, are also proved in \cite{BGG}. See also \cite{BG} for related higher order Poincar\'e-Hardy inequalities.

\medskip
No L$^p$ analogue of \eqref{L2} is known for $p\not=2$. It is our purpose here to initiate a study of \it improved $p$-Poincar\'e inequalities \rm on $\hn$, where we take the attitude of looking for improvements of the L$^p$-gap inequality
\begin{equation}\label{p-gap}
\int_{{\mathbb H}^N}|\nabla_{\hn} u|^p\,{\rm d}v_{\hn}\ge \left( \frac{N-1}{p} \right)^p\int_{{\mathbb H}^N}|u|^p\,{\rm d}v_{\hn},
\end{equation}
valid for all $u\in C_c^\infty(\hn)$, where it is known that the constant $\left( \frac{N-1}{p} \right)^p$ is the best one for such an inequality to hold, see \cite{Ngo} (a simpler proof of this fact will anyway be given below in Lemma \ref{bottom}).

In fact, let  $-\Delta_{p,\hn}$ denote the $p$-Laplacian operator on $\hn$, namely
\begin{equation}\label{expression_laplacian}
\Delta_{p, \hn}u  :=  \mbox{div}_{\hn} (|\nabla_{\hn} u|^{p-2} \nabla_{\hn} u)
\end{equation}

It is well-known that $\hn$ is a p-hyperbolic manifold, i.e., $- \Delta_{p, \hn}$ admits a positive Green's function by which the validity of a Hardy-type inequality follows. Less evident is the answer to the following question:

\vspace{.3truecm}\noindent{\it Problem.}  Does there exist a nonnegative, not identically zero weight $W$ such that the following improved Poincar\'e inequality
\begin{equation}\label{question2}
\int_{{\mathbb H}^N}|\nabla_{\hn} u|^p\,{\rm d}v_{\hn}- \left(\frac{N-1}{p} \right)^{p} \int_{{\mathbb H}^N} |u|^p\,{\rm d}v_{\hn}\ge \int_{{\mathbb H}^N}W\,|u|^p\,{\rm d}v_{\hn}
\end{equation}
holds for all $u\in C_c^\infty(\hn)$?
\vspace{.3truecm}
\rm
\par

A first affirmative answer to the above question was given in \cite{BMR}, see formula (5.25) there. In fact, the authors prove the following result:

 \begin{prop}[\cite{BMR}]\label{weight1}
Let $p >1$ and $N \geq  2$.  Set $r:=\varrho(x,x_0)$ with $x_0\in{\mathbb H}^N$ fixed. There exists a radial weight $0 <W=W(r)$ such that for all $u \in C_{c}^{\infty}(\hn)$ there holds

\begin{equation*}\label{pphardy}
\int_{\hn} |\nabla_{\hn} u |^{p} \ {\rm d}v_{\hn}  - \left( \frac{N-1}{p} \right)^p \int_{\hn} |u|^{p} \ {\rm d}v_{\hn}  \geq  \int_{\hn} W |u|^{p} \ {\rm d}v_{\hn}\,.
\end{equation*}
Furthermore,
\begin{itemize}

\item near $x_0$ there holds
\begin{equation}\label{newton}
W(r) \;\substack{\sim\\ r \rightarrow 0}\;
  \left\{
    \begin{array}{ll}
      \; \left( \dfrac{N-p}{p} \right)^p\, \dfrac{1}{r^p}& \quad \text{if } N > p\,,\\[2mm]
      \left(\dfrac{N-1}{N}\right)^N\frac1{r^N\left(\log \frac1r\right)^N} &\quad \text{if } N =p\,,\\[2mm]
     C\,\dfrac{1}{r^{\frac{p(N-1)}{p-1}}}&  \quad \text{if }N < p  \,,
    \end{array}
  \right.
\end{equation}
where $C=C(p,N) : = \left(\frac{p-1}{p} \right)^{p} \left( \int_{0}^{\infty} (\sinh s)^{-\frac{N-1}{p-1}} \ \emph{d}s \right)^{-p}$ for $N<p$.
\item Near infinity, there holds
$$ W(r) = \Lambda_p \frac{(N-1)p}{2(N-1+2(p-1))}\sinh(r)^{-2} + o(e^{-3r}) \quad \text{ as } r \rightarrow \infty.$$

\end{itemize}

\end{prop}

 Hence, the given improvement of the Poincar\'e inequality is stated in terms of a weight which is power-like near a given pole but exponentially decaying at infinity. 

\medskip

 In the present paper we construct different examples of weights $W$ for which inequality \eqref{question2} holds and that are slowly decaying at infinity. In any case, due to their asymptotic behavior the weights provided are not globally comparable. For instance, we prove the existence of a weight which is bounded but does not globally vanish at infinity. Finally, in a suitable range of $p$ we improve the Poincar\'e inequality via the Hardy weight $W=\frac{C}{\varrho^p(x,x_0)}$, where $\varrho(x,x_0)$ is the geodesic distance from $x_0\in \hn$ fixed and $C=C(N,p)$ is a positive constant. This choice seems to be the best compromise to capture the non euclidean behavior of inequality \eqref{question2} at infinity without losing too much information at the origin.
An uncertainty principle Lemma for the shifted Laplacian then follows immediately.
The techniques applied in the proofs are: hyperbolic symmetrization and p-convex inequalities together with a suitable transformation which uncovers the Poincar\'e term. Furthermore, super-solution technique and potential inequalities have been exploited.

\vskip8pt
 The paper is organized as follows. In Section 2 we state our main results on $\hn$, Theorems \ref{thphardyMaz}, \ref{mainphardy} and \ref{superphardy} and,
as a byproduct, an improved Uncertainty Principle Lemma in Corollary \ref{upl}. 
Section \ref{MazSec} discusses a related result in the Euclidean half-space, which is the key one to prove some of the results valid on $\hn$ but can have some independent interest, see Theorem \ref{Tid-remake}. 
Section 4 contains, for the convenience of the reader, a concise proof of Proposition \ref{weight1}. Section 5 discusses the proofs of Theorem \ref{Tid-remake} and, consequently, of Theorem~\ref{thphardyMaz}, which is an improvement of the Poincar\'e inequality in terms of a weight having different asymptotics in different ``directions'' and, in particular, not vanishing everywhere at infinity. Theorem \ref{mainphardy}, which states a \it Hardy-type improvement \rm of the Poincar\'e inequality in the spirit of \cite{AK}, \cite{BGG}, is proven in Section 6.
Our final result, Theorem \ref{superphardy}, deals with a related weighted inequality on the whole $\hn$. Even if it is not a direct improvement of the Poincar\'e inequality for $p\neq2$, it has an independent interest in itself due to the asymptotic behavior of the involved weight. It is proved in Section \ref{figSection}, where
as byproduct we obtain a Poincar\'e type inequality on geodesic balls.

 \section{Preliminaries and results}

We have mentioned before that inequality \eqref{p-gap} holds, and that the constant
\begin{equation}\label{LAMBDAp}
\Lambda_{p}:=\left( \frac{N-1}{p} \right)^{p}\,
\end{equation}
 appearing there is optimal. This is in fact a particular case of the work given in \cite{Ngo}, but we provide a simple proof below for the convenience of the reader.

\begin{lem}\label{bottom}
Let $N \geq 2$, $p >1$ and set $\Lambda_{p}$ as in \eqref{LAMBDAp}. There holds

\begin{equation}\label{pbottom}
\inf_{u \in W^{1,p}(\hn) \setminus \{ 0\}} \frac{\int_{\hn} |\nabla_{\hn} u|^{p} \ {\rm d}v_{\hn}}{\int_{\hn} |u|^{p} \ {\rm d}v_{\hn}}=\Lambda_{p}\,.
\end{equation}

\end{lem}
\begin{proof}
Considering the upper half space model for $\hn$, namely $\mathbb{R}^{N}_{+} = \{ (x, y) \in \mathbb{R}^{N-1} \times \mathbb{R}^{+} \} $ endowed with the Riemannian metric
$ g_{ij} = \frac{\delta_{ij}}{y^2}$
and using the expression of $p$-Laplacian \eqref{expression_laplacian} in these coordinates we have

\[
\Delta_{p,\hn} u =  y^{N} \partial_{i} ( y^{p-N} |\nabla u|^{p-2} \partial_{i} u).
\]

By computing
$-\Delta_{p, \hn}$ for the function $\rho(x, y) : = y^\alpha\in W^{1,p}_{loc}(\hn )$ where $\alpha :=\frac{N-1}{p-1}$, one has

\[
-\Delta_{p, \hn} \rho = \alpha^{p-2} \alpha (N -1 - \alpha(p -1)) y^{\alpha(p-1)}=0.
\]
Now we are in the position to apply Theorem 2.1 of \cite{Dambrosio}, obtaining

$$\int_{\hn} |\nabla_{\hn} u|^{p} \ {\rm d}v_{\hn} \geq \left(\frac{p-1}{p}\right)^p
 \int_{\hn} |u|^{p} \frac{|\nabla_{\hn} \rho|^{p}}{\rho^p}\ {\rm d}v_{\hn} =
 \Lambda_p \int_{\hn} |u|^{p} \ {\rm d}v_{\hn}$$
for all $u \in C_c^\infty(\hn)$ and hence, by density, for all $u \in W^{1,p}(\hn)$.

On the other hand, for $\varepsilon>0$, set
$$U_{\varepsilon}(x,y)=\left(\frac{y}{(1+y)^2+|x|^2}\right)^{\frac{N-1+\varepsilon}{p} }\,.$$
Since in the coordinates $(x,y)$ the volume element reads ${\rm d}v_{\hn}=\frac{{\rm d}x\, {\rm d}y}{y^N}$ and $\nabla_{\hn} u=y^2 \nabla u$, we get
\begin{align*}
\int_{\hn} | U_{\varepsilon}|^{p} \ {\rm d}v_{\hn} =\int_{\mathbb R^{+}} \int_{\mathbb{R}^{N-1}} \left(\frac{y}{(1+y)^2+|x|^2}\right)^{N-1+\varepsilon }\, \frac{{\rm d}x\,{\rm d}y}{y^N}
\end{align*}
and
$$\int_{\hn} |\nabla_{\hn} U_{\varepsilon}|^{p} \ {\rm d}v_{\hn} $$
$$= \left(\frac{N-1+\varepsilon}{p} \right)^p\,\int_{\mathbb R^{+}} \int_{\mathbb{R}^{N-1}} \left(\frac{(1-y^2+|x|^2)^2+4|x|^2y^2}{((1+y)^2+|x|^2)^2}\right)^{p/2} \left(\frac{y}{(1+y)^2+|x|^2}\right)^{N-1+\varepsilon }\, \frac{{\rm d}x\, {\rm d}y}{y^{N}}$$
$$\leq \left(\frac{N-1+\varepsilon}{p} \right)^p\,\int_{\mathbb R^{+}} \int_{\mathbb{R}^{N-1}}  \left(\frac{y}{(1+y)^2+|x|^2}\right)^{N-1+\varepsilon }\, \frac{{\rm d}x\, {\rm d}y}{y^{N}}$$
Hence, $U_{\varepsilon}(x,y)\in W^{1,p}(\hn)$ for $\varepsilon>0$ and
$\frac{\int_{\hn} |\nabla_{\hn} U_{\varepsilon}|^{p} \ {\rm d}v_{\hn}}{\int_{\hn} |U_{\varepsilon}|^{p} \ {\rm d}v_{\hn}} \leq \left(\frac{N-1+\varepsilon}{p} \right)^p$.
By letting $\varepsilon \rightarrow 0$, this argument completes the proof of the lemma.
\medskip
\end{proof}
Now we are in a situation to state our main results.

\medskip

In first place, by exploiting the half-space model for $\hn$ and following the approach of \cite{Tid}, here below we provide a weight that does not globally decay at infinity but which is bounded near $x_0$. Hence, this choice turns out to be best suited to capture the non euclidean behavior of $\hn$ which occurs at infinity. More precisely, we prove

\begin{thm}\label{thphardyMaz}
Let $p >1$, $N \geq  2$ and set $\Lambda_{p}$ as in \eqref{LAMBDAp}. There exists a bounded weight $0 <V\leq 1$ such that for all $u \in C_{c}^{\infty}(\hn)$ there holds

\begin{align}\label{phardyMaz}
\int_{{\mathbb H}^N}|\nabla_{\hn} u|^p\,{\rm d}v_{\hn}-\Lambda_p  \int_{{\mathbb H}^N} |u|^{p} \ {\rm d}v_{\hn}
 \geq
   \left( \frac{N-1}{p} \right)^{p-2}  C(N,p)   \int_{{\mathbb H}^N} V\,|u|^p \ {\rm d}v_{\hn},
\end{align}
where $C(N,p)$ is a positive constant that can be explicitily computed for which the following estimates hold
\begin{equation}
  \label{cnp}
\begin{aligned}
  C(N,p)\ge&\frac{1}{4p'}, & if\ 1<p\le 4/3,\\
  C(N,p)\ge&\left(2(8-3p)+2\sqrt{p'(8-3p)}  \right)^{-1}, & if\ 4/3<p\le 2, \\
  C(N,p)=&\frac{1}{\sqrt 2}\, \frac{1}{\sqrt 2\, p + 2 \sqrt p},  & if\  2< p\le 2(N-1)^2, \\
  C(N,p)=&\left(\frac{p}{N-1}+2p+2(N-1)\right)^{-1},& if\   p> 2(N-1)^2,  \\
\end{aligned}
\end{equation}
where $p'>1$ denotes the conjugate exponent of $p$.
\medskip

Furthermore, set $r:=\varrho(x,x_0)$ with $x_0\in{\mathbb H}^N$ fixed, we have
\begin{itemize}
\item for any $0<\alpha\leq 1$ there exists an unbounded set $U_{\alpha}\subset \hn$ such that $V\vert_{U_{\alpha}}\equiv\alpha$ and $U_{\alpha} \cap (B(x_0,2r)\setminus B(x_0,r)) \neq \emptyset$ as $r\rightarrow +\infty$;
\item for any $\beta>0$ there exists an unbounded set $W_{\beta}\subset \hn$ such that $V\vert_{W_{\beta}}\sim \sqrt{\frac{\beta}{2}}\,e^{-r/2}$ as $r\rightarrow +\infty$.
\end{itemize}
\end{thm}
It is worth noticing that the weight $V$ can be written, in the half-space model, as $V(x_1,...,x_{N-1},y):=\frac{y}{\sqrt{y^2+x_1^2}}$, see Theorem \ref{Tid-remake} in Section \ref{MazSec} from which the above statements follow.

\medskip

Even if both the inequalities provided by Proposition \ref{weight1} and Theorem \ref{thphardyMaz} are of the form \eqref{question2} they seem to lose too much information, respectively,  at infinity or near the origin. To this aim, a good compromise is represented by the following Poincar\'e-Hardy inequality

\begin{thm}\label{mainphardy}
Let $p \geq  2$ and  $N \geq 1 + p(p -1).$ Set $\Lambda_p$ as in \eqref{LAMBDAp} and $r:=\varrho(x,x_0)$ with $x_0\in{\mathbb H}^N$ fixed. Then for $u \in C_{c}^{\infty}(\hn)$ there holds

\begin{equation}\label{eqphardy}\begin{aligned}
\int_{{\mathbb H}^N}|&\nabla_{\hn} u|^p\,{\rm d}v_{\hn}-\Lambda_p  \int_{{\mathbb H}^N} |u|^{p} \ {\rm d}v_{\hn}\\
 &\geq
   (p-1) \left( \frac{N-1}{p} \right)^{p-2} \left( \frac{p-1}{p} \right)^2  \int_{{\mathbb H}^N} \frac{|u|^p}{r^p} \ {\rm d}v_{\hn}.
\end{aligned}\end{equation}
\end{thm}

\begin{rem}
From the above theorem, we can easily infer that the best constant in the r.h.s. of \eqref{eqphardy}, i.e.
\[
c_{p} :
= \inf_{C_{c}^{\infty}(\hn) \setminus \{ 0 \}} \frac{\int_{{\mathbb H}^N}|\nabla_{\hn} u|^p\,{\rm d}v_{\hn}-
\Lambda_p  \int_{{\mathbb H}^N} |u|^{p} \ {\rm d}v_{\hn}  }{\int_{{\mathbb H}^N} \frac{|u|^p}{r^p} \ {\rm d}v_{\hn}}\,,
\]
blows up as $N \rightarrow \infty$ if $p>2$. This does not happen in the linear case $p = 2,$ where $c_{2} = \frac{1}{4}$, see \eqref{L2}, where it is known that the constant $c_2$ is optimal. This issue was proved in \cite{BGG} by providing an explicit super-solution for the corresponding Euler-equation, a construction that also allows to determine a remainder term for \eqref{L2} of the type $\frac{1}{\sinh^2 r}$, see Remark \ref{c2}. Unfortunately, this argument carries over to the case $p>2$ only partially thereby allowing to prove Theorem \ref{hardyball} on suitable geodesic balls.
\end{rem}

As an immediate consequence of the previous result one gets the following \it uncertainty principle \rm for the quadratic form of the shifted Laplacian. For a similar result, when $p=2$, concerning the quadratic form of the Laplacian, see \cite[Theorem 4.1]{Kombe2}.

\begin{cor}\label{upl}
Let $p \geq  2$ and  $N \geq 1 + p(p -1).$ Set $\Lambda_p$ as in \eqref{LAMBDAp} and $r:=\varrho(x,x_0)$ with $x_0\in{\mathbb H}^N$ fixed. Then for $u \in C_{c}^{\infty}(\hn)$ there holds:
\begin{equation}\begin{aligned}
\left[\int_{{\mathbb H}^N}|\nabla_{\hn} u|^p\,{\rm d}v_{\hn}-\Lambda_p  \int_{{\mathbb H}^N} |u|^{p} \ {\rm d}v_{\hn}\right]\,\left[\int_{{\mathbb H}^N} |u|^{p}\,r^{p'}\ {\rm d}v_{\hn}\right]^{\frac p{p'}}\\
\ge (p-1) \left( \frac{N-1}{p} \right)^{p-2} \left( \frac{p-1}{p} \right)^2 \left[\int_{{\mathbb H}^N}|u|^{p}\ {\rm d}v_{\hn}\right]^p,
\end{aligned}
\end{equation}
where $p'>1$ denotes the conjugate exponent of $p$.
\end{cor}

\begin{rem}
In Theorem \ref{mainphardy}, the restrictions $p\geq 2$ and $N \geq 1 + p(p -1)$ are technical. In particular, the latter only comes from the last step in the proof. Nevertheless, the very same assumption also appears in the Poincar\'e-Hardy inequality below where the constant $\Lambda_p$
in (\ref{eqphardy}) is replaced by a non-constant weight: $\Lambda_p\,H_p(r)$. Here, $H_p(r)$ is a positive function which is larger then one in $(0,r_p)$, smaller then one in $(r_p,+\infty)$, and that converges to one as $r\to+\infty$, see Figure 1 in Section \ref{figSection}.
Since the proofs of the two theorems are completely different, we are led to believe that a deeper relation between the dimension restriction and the weight considered might exist.
\end{rem}

\begin{thm}\label{superphardy}

Let $p \geq  2$ and  $N \geq 1 + p(p -1).$ Set $\Lambda_{p}$ as in \eqref{LAMBDAp} and $r := \varrho(x, x_{0})$ with $x_0\in \hn$ fixed. Then for $u \in C_{c}^{\infty}(\hn)$ there holds

\begin{equation}\begin{aligned}\label{eqphardy2}
&\int_{{\mathbb H}^N}|\nabla_{\hn} u|^p\,{\rm d}v_{\hn}-\Lambda_p \int_{{\mathbb H}^N} H_{p}(r) |u|^{p} \ {\rm d}v_{\hn}   \geq\\ &\frac{(p-1)^{p-1} (N(p -2) + 1)}{p^p}   \int_{{\mathbb H}^N} \frac{|u|^p}{r^p} \ {\rm d}v_{\hn} \\
 &+ \frac{(N-1)(N-1 -p(p-1))(p - 1)^{p-2}}{p^{p}} \int_{\hn} \frac{|u|^{p}}{\sinh^{p} r} \ {\rm d}v_{\hn}
\end{aligned}
\end{equation}
where $H_{p}(r) = \left( \coth r - \left(\frac{p-1}{N-1}\right)  \frac{1}{r} \right)^{p -2}.$
\end{thm}

\begin{rem}\label{c2}
When $p=2$, the statement of Theorem \ref{superphardy} includes that of Theorem \ref{mainphardy} providing a further remainder term.
Unfortunately, the weight $H_p$ is larger than one only for $r$ small, hence \eqref{eqphardy2} is not an improvement of the $p$-Poincar\'e inequality if $p\not=2$. Nevertheless, for functions having support outside large balls the inequality becomes very "close" to the Poincar\'e one, see Lemma \ref{asymptotic}.
\end{rem}

In Section \ref{figSection}, from Theorem \ref{superphardy}, we deduce
an inequality
involving the same weight of (\ref{eqphardy}) but holding on geodesic balls.

 \section{Related Hardy-Maz'ya-type Inequalities on Half-space}\label{MazSec}

 This section is devoted to the study of improved Hardy-Maz'ya-type inequalities on upper half space. There have been an extensive research on Hardy-Maz'ya inequality
 (see \cite{FMT, FTT,mancini,Ma}).  Our main goal here is to present some Hardy-Maz'ya inequalities strictly related to our Poincar\'e-Hardy inequalities on the hyperbolic space. We begin with the counterpart of Lemma \ref{bottom}:

  \begin{lem} \label{lemMaz}
 Let $p>1$, $N \geq 2$ and set $\Lambda_{p}$ as in \eqref{LAMBDAp}. Then for all $u \in C_{c}^{\infty}(\mathbb{R}^{N}_{+})$ there holds

 \begin{equation}\label{mazya-type}
\int_{\mathbb R^{+}} \int_{\mathbb{R}^{N-1}} \frac{|\nabla u|^p}{y^{N-p}} \ {\rm d}x \ {\rm d}y \geq  \Lambda_p \int_{\mathbb{R}^{+}} \int_{\mathbb{R}^{N-1}} \frac{|u|^p}{y^N} \ {\rm d}x \ {\rm d}y\,,
\end{equation}
where $\nabla u$ denotes the euclidean gradient. Moreover the constant $\Lambda_{p}$ appearing in \eqref{mazya-type} is sharp.
 \end{lem}

\begin{proof}
The proof of Lemma \ref{lemMaz} follows by noticing that in the upper half space model for $\hn$, see the proof of Lemma  \ref{bottom}, \eqref{pbottom} readily writes as the Hardy-Maz'ya-type inequality \eqref{mazya-type}. Hence, the statement of Lemma \ref{lemMaz} comes as a corollary of Lemma \ref{bottom}.
\end{proof}

Next we turn to the main result of this section. We improve \eqref{mazya-type} by providing a suitable remainder term.

\begin{thm}\label{Tid-remake}
Let $p >1$, $N \geq 2$ and set $\Lambda_{p}$ as in \eqref{LAMBDAp}. For all $u \in C_{c}^{\infty}(\mathbb{R}^{N}_{+})$ there holds
\begin{equation}\begin{aligned}\label{phardyTid}
\int_{\mathbb R^{+}} \int_{\mathbb{R}^{N-1}} \frac{|\nabla u|^p}{y^{N-p}} \ {\rm d}x \ {\rm d}y -  \Lambda_p \int_{\mathbb{R}^{+}} \int_{\mathbb{R}^{N-1}} \frac{|u|^p}{y^N} \ {\rm d}x \ {\rm d}y
 \geq \\
   \left( \frac{N-1}{p} \right)^{p-2}  C(N,p)  \int_{\mathbb{R}^{+}} \int_{\mathbb{R}^{N-1}} \frac{|u|^p}{y^{N-1}\sqrt{y^2+x_1^2}} \ {\rm d}x \ {\rm d}y.
\end{aligned}\end{equation}
where $C(N,p)$ is a positive constant as in (\ref{cnp}).
\end{thm}

It is worth noting that Theorem \ref{thphardyMaz} turns out to be a consequence of the above theorem. We postpone the proofs of Theorem~\ref{Tid-remake} and, hence, of Theorem~\ref{thphardyMaz} to Section \ref{HipMaz}.

\section{ Proof of Proposition~\ref{weight1}}
We recall for the convenience of the reader the proof given in \cite{BMR}, only the asymptotics at infinity not being explicitly given there. The proof relies on the well known classical Hardy inequality with respect to the Green's function and exploiting its behavior on hyperbolic space.
More precisely, for $N \geq 2$ and $p >1$,
the following Hardy inequality holds (see \cite{Dambrosio}, \cite{BMR}):

\begin{equation}\label{ephardy}
\int_{\hn} |\nabla_{\hn} u |^{p} \ {\rm d}v_{\hn} \geq \left( \frac{p-1}{p} \right)^{p} \int_{\hn} \left| \frac{\nabla G_{p}}{G_{p}} \right|^{p} |u|^{p} \ {\rm d}v_{\hn},
\end{equation}
for $u \in C_{c}^{\infty}(\hn)$, where $ G_{p}$ is the Green's function of $-\Delta_{p,\hn}$ which, up to a positive multiplicative constant, is given by
\[
G_{p}(r) :=  \int_{r}^{\infty} (\sinh s)^{-\frac{N-1}{p-1}} \ \mbox{d}s.
\]
Indeed, if $p>N$, then $G_p\in W^{1,p}_{loc}(\hn)$ and hence \cite[Theorem 2.1]{Dambrosio} applies. For $1<p\le N$ the inequality (\ref{ephardy}) holds for functions
$u \in C_{c}^{\infty}(\hn \setminus \{ x_{0} \})$, and since $\{ x_{0} \}$
is a compact set of zero $p$-capacity, the claim follows from \cite[Corollary 2.3]{Dambrosio}.

The proof is then a calculus exercise involving the asymptotics of the function $G_{p}(r)$. Indeed, Eq. \eqref{ephardy} may be rewritten as
\[
\int_{\hn} |\nabla_{\hn} u |^{p} \ {\rm d}v_{\hn} - \Lambda_{p} \int_{\hn} |u|^p \ {\rm d}v_{\hn} \geq  \int_{\hn} W |u|^{p} \ {\rm d}v_{\hn},
\]
where $$W(r) :=  \left( \frac{p-1}{p} \right)^{p} \left| \frac{G_{p}'(r)}{G_{p}(r)} \right|^{p} - \Lambda_{p}\,,$$
with $\Lambda_{p}$ as in \eqref{LAMBDAp}.

First we claim that $W> 0.$ From the expression of the Green's function we have
\begin{align*}
G_{p}(r) & = \int_{r}^{\infty} (\sinh s)^{-\frac{N-1}{p-1}} \ \mbox{d}s = \int_{r}^{\infty} (\sinh s)^{-\frac{N-1}{p-1} -1} \sinh s \ \mbox{d}s \\
& < \int_{r}^{\infty} (\sinh s)^{-\frac{N-1}{p-1} -1} \cosh s \ \mbox{d}s = \int_{\sinh r}^{\infty} t^{- \frac{N-1}{p-1} - 1} \ \mbox{d}t \\
& = \frac{p-1}{N-1} (\sinh r)^{-\frac{N-1}{p-1}}.
\end{align*}
Moreover, we also have $G_{p}^{\prime}(r) = - (\sinh r)^{-\frac{N-1}{p-1}}.$ Therefore,

\[
\left| \frac{ G_{p}^{\prime}(r)}{G_{p}(r)} \right|^{p} > \left( \frac{N-1}{p-1} \right)^p,
\]
and hence this proves  $\left( \frac{p-1}{p} \right)^{p} \left| \frac{ G_{p}^{\prime}(r)}{G_{p}(r)} \right|^{p} > \Lambda_p.$

Let us turn to study the asymptotic behavior of $W$ near the origin. First consider the case when $N \geq p.$ Then, $G_{p}(r) \rightarrow \infty$ as
 $r\rightarrow 0$ and, using de L'H\^opital's rule, we obtain:
 $$
 \lim_{r \rightarrow 0} \frac{r\, G_{p}^{\prime}(r)}{G_{p}(r)}= \frac{p-N}{p-1}\quad \text{if }N>p\,
 $$
and
 $$
 \lim_{r \rightarrow 0} \frac{r\,\log r\, G_{p}^{\prime}(r)}{G_{p}(r)}= 1\quad \text{if }N=p\,.
 $$

 Whence, the stated asymptotics easily follows.

When $N < p,$ in the second term above one has $\int_{r}^{\infty} (\sinh s)^{-\frac{N-1}{p-1}} \ \mbox{d}s < \infty$ as $ r \rightarrow 0$. Hence,  \eqref{newton} follows immediately by exploiting
$\sinh r \sim  r$ as $r \rightarrow 0.$

\medskip
Finally, we study the asymptotics of $W$ near infinity. For this we note that
\begin{align*}
G_{p}(r) & = \int_{r}^{\infty} (\sinh s)^{-\frac{N-1}{p-1}} \ ds = \int_{\sinh r}^{\infty} t^{-\frac{N-1}{p-1} } (1 + t^2)^{-\frac{1}{2}} \ dt \\
& = \int_{\sinh r}^{\infty} t^{-\frac{N-1}{p-1} -1} \left[ 1 - \frac{1}{2t^2} + o \left( \frac{1}{t^3} \right) \right] \ dt, \  \quad  r \rightarrow \infty \\
& = \frac{p-1}{N-1} (\sinh r)^{-\frac{N-1}{p-1}} - \left(2\frac{N-1}{p-1}+4 \right)^{-1} (\sinh r)^{-\frac{N-1}{p-1} -2} + o \left( (\sinh r)^{-\frac{N-1}{p-1} -3}  \right) ,
\end{align*}
hence we have

\begin{align*}
\left| \frac{G_{p}^{\prime}(r)}{G_{p}(r)} \right|^{p} &=
\left| \frac{p-1}{N-1} - \left(2\frac{N-1}{p-1}+4 \right)^{-1} (\sinh r)^{ -2} + o \left( (\sinh r)^{ -3}  \right)  \right|^{-p}=\\
&= \left(\frac{N-1}{p-1} \right)^{p}\left( 1 +
\frac{p\frac{N-1}{p-1}}{2(\frac{N-1}{p-1}+2)} (\sinh r)^{-2} + o ((\sinh r)^{-3})\right).
\end{align*}
This completes the proof.

\section{Proof of Theorem~\ref{Tid-remake} and Theorem~\ref{thphardyMaz}}\label{HipMaz}

\noindent{\bf {Proof of Theorem~\ref{Tid-remake}}}

\medskip

The key ingredients in the proof are the following Lemma \ref{Tid0} from \cite{Tid} that we adapt to our situation with a suitable choice of the parameters, and the inequality \eqref{ni} which represents an improvement of the analogous inequalities presented in \cite{Tid}.
\begin{lem}\label{Tid21} \cite[Lemma 2.1]{Tid}
Let $\Omega$ be a convex domain in $\mathbb{R}^{N}$ and set $\delta(z):=$dist$(z,\partial \Omega)$ for any $z\in \Omega$. Let $d\in(-\infty,mp-1)$ where $m\in \mathbb{N}_+$ and let ${\textbf F}=(F_1, . . . , F_N)$ be a $C^1(\Omega)$ vector field in $\mathbb{R}^{N}$. Furthermore, let $w \in C^1(\Omega)$ be a nonnegative weight function and
$$h_{p,m,d}: = \left(\frac{mp-d-1}{p} \right)^p\,.$$
Then, the following inequality holds
\begin{equation}\begin{aligned}\label{Tid0}
\int_{\Omega}  \frac{|\nabla u|^p\, w}{\delta^{(m-1)p-d}} \ dz \geq  h_{p,m,d} \left( \int_{\Omega}  \frac{| u|^p\, w}{\delta^{mp-d}} - \frac{p|u|^p \Delta \delta\, w}{(mp-d-1)\delta^{mp-d-1}} \ dz \right) \\
+  h_{p,m,d}  \int_{\Omega} \left[   \frac{p\,\Div {\textbf F}}{mp-d-1} +  \frac{p-1}{\delta^{mp-d}} \left(1-|\nabla \delta-\delta^{mp-d-1}{\textbf F}|^{\frac{p}{p-1}}\right)\right]|u|^pw \,dz  \\
+\left(\frac{mp-d-1}{p} \right)^{p-1}  \int_{\Omega} \nabla w \cdot \left({\textbf F} -\frac{\nabla \delta}{\delta^{mp-d-1}} \right) |u|^p \,dz\,,
\end{aligned}\end{equation}
for all $u \in C_{c}^{\infty}(\Omega)$.
\end{lem}

We will apply Lemma \ref{Tid21} with $\Omega=\mathbb R_+^{N}$. Hence, $z=(x_1,...,x_{N-1},y)=(x,y)$ with $x\in \mathbb R^{N-1}$, $y\in \mathbb R^{+}$, and $\delta(z)=y$. Furthermore, we fix $w=1$, $m=2$ and $d=mp-N$ so that $d<mp-1$ for any $p\geq 1$ and $N >1$ and we obtain $h_{p,m,d}=\Lambda_p$. Then, \eqref{Tid0} reads as follows.

\begin{lem}\label{Tid21BIS}
Let $p > 1$, $N\geq 2$ and set $\Lambda_{p}$ as in \eqref{LAMBDAp}. For any $C^1(\mathbb R_+^{N})$ vector field ${\mathbf F}=(F_1,...,F_N)$, the following inequality holds
\begin{equation}\begin{aligned}\label{Tid1}
\int_{\mathbb R^{+}} \int_{\mathbb{R}^{N-1}} \frac{|\nabla u|^p}{y^{N-p}} \ {\rm d}x \ {\rm d}y -  \Lambda_p \int_{\mathbb{R}^{+}} \int_{\mathbb{R}^{N-1}} \frac{|u|^p}{y^N} \ {\rm d}x \ {\rm d}y
 \geq \\
  \Lambda_p  \int_{\mathbb{R}^{+}} \int_{\mathbb{R}^{N-1}} \left[ \frac{p\, \Div {\mathbf F}}{N-1}  +  \frac{p-1}{y^N}\left(1-|(0,...,0,1)-y^{N-1}{\mathbf F}|^{\frac{p}{p-1}}\right) \right]|u|^p \ {\rm d}x \ {\rm d}y\,,
\end{aligned}\end{equation}
for all $u \in C_{c}^{\infty}(\mathbb{R}^{N}_{+})$.

\end{lem}

\begin{lem} Let $b>0$ and $s\in [0,1]$ then
  \begin{equation}  \label{ni}
  1-(1-s)^b\ge b s - q_b(b-1)s^2
\end{equation}
where
\begin{equation}
  q_b:=
  \begin{cases}
    1  & if\ 1\le b \le 2; \\
    b/2 & if\   0 <b<1\ or \ 2<b.
  \end{cases}
\end{equation}
\end{lem}
\begin{proof} Taylor expansion of $(1-s)^b$ around 0 gives
$(1-s)^b=1-bs+\frac b2 (b-1)s^2 +R(s)$ where the reminder term $R(s)$ is given by
$R(s)=-s^3 b (b-1)(b-2) (1-t)^{b-3}/6$ with a suitable $t\in[0,s]$.
For $s\in [0,1]$ and $b\ge 2$ or $0<b\le 1$, $R(s)\le 0$ and the claim follows.

For the case $1<b<2$ the claim will follow by proving that the function
$g(s):=(1-s)^b-1+bs-(b-1)s^2$ is nonpositive on $[0,1]$.
To this end since $g'''>0$ one deduces that $g''$ is negative on an interval
$]0,s_0[$ and positive on $]s_0,1[$, which in turn, with the fact that $g'(0)=0$ and $g'(1)>0$,  implies that $g'$ has only a critical point on $]0,1[$.
Since $g(0)=g(1)=0$ and $g'(1)>0$ we obtain that the maximum of $g$ is 0.
\par

\end{proof}

For sake of brevity we introduce the following notation
 $$I(u):=  \int_{\mathbb{R}^{+}} \int_{\mathbb{R}^{N-1}}\frac{|\nabla u|^p}{y^{N-p}}\ {\rm d}x \ {\rm d}y-\Lambda_p \int_{\mathbb{R}^{+}} \int_{\mathbb{R}^{N-1}}\frac{|u|^p}{y^{N}}\ {\rm d}x \ {\rm d}y,$$
and
$$w:= \frac{y}{\sqrt{y^2+x_1^2}}.$$

\ Next, in the spirit of \cite[Theorem 4.1]{Tid}, for any $0\leq a \leq 1$ we write \eqref{Tid1} with
${\mathbf F}_1:= \left(0,...,\frac{aw}{y^{N-1}} \right)$.
Since $0\le w \le 1$ we get
\begin{equation}
  \label{eq:div1}
    \Div {\mathbf F}_1\ge (2-N)a\frac{w}{y^N}-a \frac{w^2}{y^N},
\end{equation}
and, by using (\ref{ni}) with $b=p'$ and the fact that $0\le aw\le 1$, we have
\begin{equation}
  \label{disf1}
  1-|(0,\dots,1)-y^{N-1}{\mathbf F}_1|^{p'}=1-(1-aw)^{p'}\ge p' aw - q_{p'}(p'-1)a^2w^2.
\end{equation}
By using (\ref{eq:div1}) and ({\ref{disf1}) in (5.2), the square bracket in right hand side can be estimated as
  \begin{equation}\begin{aligned}
    \label{eq:sq1}
   & \left[ \frac{p\, \Div {\mathbf F}_1}{N-1}  +  \frac{p-1}{y^N}\left(1-|(0,...,0,1)-y^{N-1}{\mathbf F}_1|^{\frac{p}{p-1}}\right) \right]\\
   &\ge a\frac{p}{N-1}\frac{w}{y^N}-a(\frac{p}{N-1}+q_{p'}a)\frac{w^2}{y^N} =:S_1
  \end{aligned}\end{equation}
Therefore, from \eqref{Tid21BIS}  we obtain
\begin{equation}
  \label{I1}
  I(u)\ge \Lambda_p\int_{\mathbb R^{+}} \int_{\mathbb{R}^{N-1}} S_1 |u|^p\ {\rm d}x \ {\rm d}y
\end{equation}
for all $u \in C_{c}^{\infty}(\mathbb{R}^{N}_{+})$.

Similarly, for any $0\leq c \leq 1$, choosing 
${\mathbf F}_2=c \left(\frac{x_1 w^2}{y^N},0,...,0,\frac{ yw^2}{y^N} \right)$,
by an explicit computation we obtain
\begin{equation}
  \label{eq:div2}
    \Div {\mathbf F}_2 = c(2-N)\frac{w^2}{y^N}
\end{equation}
and
\begin{equation}
  |(0,\dots,1)-y^{N-1}{\mathbf F}_2|^2=1-c(2-c) w^2.
\end{equation}
Evaluating the square bracket in r.h.s. of (5.2),
by using (\ref{ni}) with $b=p'/2$ and the fact $0\le c(2-c){w^2}\le 1$, we have
\begin{equation}\begin{aligned}
    \label{eq:sq2}
    \left[ \frac{p\, \Div {\mathbf F}_2}{N-1}  +  \frac{p-1}{y^N}\left(1-|(0,...,0,1)-y^{N-1}{\mathbf F}_2|^{\frac{p}{p-1}}\right) \right]\\
= \frac{p}{N-1}c(2-N)\frac{w^2}{y^N}+
    \frac{p-1}{y^N}\left(1-\left(1-c(2-c){w^2}\right)^{p'/2}\right)\\
    \ge \frac{p}{N-1} c (1-c \frac{N-1}{2}) \frac{w^2}{y^N}-(p-1)c^2(2-c)^2
      q_{p'/2}(\frac{p'}{2}-1)\frac{w^4}{y^N}
=:S_2\end{aligned}
  \end{equation}
>From Lemma 5.2 we deduce
\begin{equation}
  \label{I2}
  I(u)\ge \Lambda_p\int_{\mathbb R^{+}} \int_{\mathbb{R}^{N-1}} S_2 |u|^p\ {\rm d}x \ {\rm d}y
\end{equation}

\noindent {\bf Case $1<p\le 2$.} In this case since $0\le w\le 1$ and
$p'/2-1=\frac{(2-p)}{2(p-1)}\ge 0$ we have
$$ S_2\ge \frac{w^2}{y^N} \frac{p}{N-1} f(c),$$
where
$$ f(c):=c\left(1-c \frac{N-1}{2}\right) -c^2(2-c)^2  q_{p'/2}\frac{(2-p)(N-1)}{2p}.$$

Set $M:=\max\{f(c), c\in [0,1]\}$. Since $f(0)=0$ and $f'(0)=1>0$ we have that $M>0$. Hence we have
$$  S_2\ge M\frac{p}{N-1} \frac{w^2}{y^N},$$
which in turns yields
\begin{equation}
 \label{I22}
  I(u)\ge \Lambda_p \frac{p}{N-1} M\int_{\mathbb R^{+}} \int_{\mathbb{R}^{N-1}}  \frac{w^2}{y^N} |u|^p\ {\rm d}x \ {\rm d}y.
\end{equation}

For $1<p\le 2$, since $q_{p'}=\frac{p}{2(p-1)}$, (\ref{I1}) reads as
\begin{align} \label{I12}
   I(u)& \ge \Lambda_p\frac{p}{N-1} a \int_{\mathbb R^{+}} \int_{\mathbb{R}^{N-1}} \frac{w}{y^N} |u|^p\ {\rm d}x \ {\rm d}y \notag \\
    & -\Lambda_p\frac{p}{N-1}
     a\left(1+ \frac{N-1}{2(p-1)}a  \right)  \int_{\mathbb R^{+}} \int_{\mathbb{R}^{N-1}} \frac{w^2}{y^N} |u|^p\ {\rm d}x \ {\rm d}y.
\end{align}
Multiplying (\ref{I22}) by $\frac aM\left(1+ \frac{N-1}{2(p-1)}a  \right)$ and summing up to  (\ref{I12}) we have
\begin{equation} \label{eq:1bis}
   I(u)\ge \Lambda_p\frac{p}{N-1} {\mu_1(a)} \int_{\mathbb R^{+}} \int_{\mathbb{R}^{N-1}} \frac{w}{y^N} |u|^p\ {\rm d}x \ {\rm d}y,
\end{equation}
where
  $$\mu_1(a):=\frac{a}{1+\frac aM \left(1+ \frac{N-1}{2(p-1)} a\right)}. $$

 Setting $C(N,p):=\frac{N-1}{p}\max\{\mu_1(a), a\in[0,1]\}$ we get the claim.

Now we proceed to obtain an explicit estimate on $C(N,p)$.
To this end we first look for some bounds on $M=\max\{f(c), c\in [0,1]\}$.
Since  $c\ge0$ and $(2-p)\ge0$
from the chain of inequalities
$$f(c)\le  c\left(1-c \frac{N-1}{2}\right)\le \frac{1}{2(N-1)},$$
we deduce
\begin{equation}  \label{M+}
   M \le\frac 1{2}.
\end{equation}

Next step is to estimate the maximum of $\mu_1$. The function $\mu_1(a)$
for $a\ge 0$ attains its maximum at $a_0:=\sqrt{\frac{2(p-1)}{N-1} M}$.
>From the bound $M\le 1/2$, we immediately deduce that $0<a_0\le1$,
and hence
$$C(N,p)=\frac{N-1}{p}\mu_1(a_0)=\frac{N-1}{p} \frac{M}{1+
\sqrt{\frac{2(N-1)M}{p-1}}}=:\gamma(M).
$$

Since $\gamma$ is increasing, a bound from below on $M$ yields a bound from below on $C(N,p)$.
Set $\beta:=N-1$ and $\delta:=q_{p'/2}\frac{2-p}{p}$.
For $0\le c\le 1$, $f(c)$ can be estimated as
$$ f(c)=c\left(1-c \frac{\beta}{2}(1+4\delta) +2\beta\delta c^2(1-\frac14 c)\right)\ge
c\left(1-c \frac{\beta}{2}(1+4\delta)\right).$$
That is, by choosing $c_0:=\frac{1}{\beta(1+4\delta)}$, we have
$$M\ge f(c_0)=\frac{1}{2\beta(1+4\delta)}, $$
and hence
\begin{equation}
  \label{eq:cnp-}
  C(N,p)=\gamma(M)\ge\gamma\left(\frac{1}{2\beta(1+4\delta)}\right)=
  \frac{1}{2p(1+4\delta)}\ \frac{1}{1+\left((p-1)(1+4\delta)\right)^{-1/2}}.
\end{equation}

Now,  taking into account
that
for $1<p\le 4/3$ one has $q_{p'/2}=\frac{p'}{4}$, while for
$4/3<p\le 2$ one gets $q_{p'/2}=1$, plugging $\delta=\frac{2-p}{p}q_{p'/2}$
in (\ref{eq:cnp-}), we obtain the estimates.

\bigskip

\noindent {\bf Case $p > 2$.}
In this case we have for any $c\in [0,1]$
\begin{eqnarray} \label{sq22}
  S_2&\ge& \frac{p}{N-1} c (1-c \frac{N-1}{2}) \frac{w^2}{y^N}-(p-1)c^2(2-c)^2
      q_{p'/2}\frac{2-p}{2}\frac{w^4}{y^N}\\
      &\ge&  \frac{p}{N-1} c (1-c \frac{N-1}{2}) \frac{w^2}{y^N}.
\end{eqnarray}
Choosing $c=1/(N-1)$ we obtain
\begin{equation}
      S_2 \ge  \frac{p}{N-1} \frac{1}{2(N-1)} \frac{w^2}{y^N},
\end{equation}
  and hence we have
\begin{equation} \label{I21}
   I(u)\ge \Lambda_p\frac{p}{N-1} \frac{1}{2(N-1)}  \int_{\mathbb R^{+}} \int_{\mathbb{R}^{N-1}} \frac{w^2}{y^N} |u|^p\ {\rm d}x \ {\rm d}y
\end{equation}
Since $1<p'\le 2$ we have that $q_{p'}=1$ and  (\ref{I1}) reads as
\begin{equation} \label{I11}\begin{aligned}
   I(u)&\ge \Lambda_p\frac{p}{N-1} a \int_{\mathbb R^{+}} \int_{\mathbb{R}^{N-1}} \frac{w}{y^N} |u|^p\ {\rm d}x \ {\rm d}y\\ &-  \Lambda_p\frac{p}{N-1}
     a\left(1+ \frac{N-1}{p}a  \right)  \int_{\mathbb R^{+}} \int_{\mathbb{R}^{N-1}} \frac{w^2}{y^N} |u|^p\ {\rm d}x \ {\rm d}y
\end{aligned}\end{equation}
Multiplying (\ref{I21}) by $2(N-1) a \left(1+ \frac{N-1}{p}a  \right)$ and using (\ref{I11}) we have
\begin{equation} \label{eq:1}
   I(u)\ge \Lambda_p\frac{p}{N-1} {\mu_2(a)} \int_{\mathbb R^{+}} \int_{\mathbb{R}^{N-1}} \frac{w}{y^N} |u|^p\ {\rm d}x \ {\rm d}y
\end{equation}
where
  $$\mu_2(a):=\frac{a}{1+2(N-1) \,a\left(1+ \frac{N-1}{p} a\right)}. $$

 Setting $C(N,p):=\frac{N-1}{p}\max\{\mu_2(a), a\in[0,1]\}$ we get the claim.

Now we proceed to compute $C(N,p)$.
The maximum of $\mu_2$ is achieved at $a_0:=\frac1{N-1}\sqrt{\frac p2}$ if $a_0\le 1$,
at 1 else. That is,
\begin{itemize}
  \item[-] if $2<p\le 2 (N-1)^2$ we have
    $C(N,p)=\frac{N-1}{p} \mu_2(a_0)= \left(\sqrt 2 (\sqrt2 p+ 2\sqrt p)\right)^{-1}$;
  \item[-] if $p> 2 (N-1)^2$ we have
    $C(N,p)=\frac{N-1}{p} \mu_2(1)= \frac{N-1}{p}  \left( 1+2(N-1)+2\frac{(N-1)^2}{p}\right)^{-1}$.
\end{itemize}
This concludes the proof of Theorem~\ref{Tid-remake}.

\begin{rem}\label{rem2} Let $1<p<2$. Here, we compute $C(2,p)$, that is when $N=2$.
  In this case, with the same notation used in the proof of  Theorem~\ref{Tid-remake}, the function $f$ reads as
$$ f(c)=c\left(1- \frac{1}{2} c -\frac 12 c(2-c)^2\delta  \right).$$

Consider first the case $4/3\le p<2$. In this case $\delta\in ]0,1/2]$ and the only critical point of $f$ in $ [0,1]$ is at $c=1$, therefore $f$ attains its maximum at 1,
that is $M=f(1)=(1-\delta)/2<1/2$. Therefore, by definition of $C(2,p)$ we have
$$C(2,p)=\frac 1p\, \frac{(1-\delta)/2}{1+\sqrt{\frac{1-\delta}{4(p-1)}}}
  =\frac{1}{p'}\, \frac{\sqrt 2}{\sqrt 2 p+ \sqrt p}.
$$

Next we consider the case $1<p < 4/3$. Now we have $\delta\in]1/2,+\infty[$ and
the function $f$ has in $[0,1]$ two distinct critical value
$c_0= 1-\sqrt{1-\frac1{2\delta}}$ and $c_1=1$.
Since $f''(1)=2\delta -1 >0$, the maximum is attained at $c_0$,
that is $M=f(c_0)=\frac1{8\delta} (<1/4)$. Therefore
$$C(2,p)=\frac 1p\, \frac{(1/8\delta)}{1+\sqrt{\frac{1/8\delta}{2(p-1)}}}
  =\frac1{p'}\, \frac 1{2(2-p)+\sqrt{2-p}}.$$
\end{rem}

\par \bigskip\par

\noindent\textbf{Proof of Theorem \ref{thphardyMaz}}
\medskip

Letting $V(x_1,...,x_{N-1},y):=\frac{y}{\sqrt{y^2+x_1^2}}$, the proof of \eqref{phardyMaz} follows at once from \eqref{phardyTid} by exploiting the half-space model for $\hn$ as explained in the proof of Lemma \ref{bottom}. Next, for any $\alpha\in (0,1]$, set $U_{\alpha}:=\{(x, y) \in \mathbb{R}^{N}_{+} : x_1=ky\text{ with } k^2=(1-\alpha^2)/\alpha^2 \}$. Clearly, $V\vert_{U_{\alpha}}\equiv\alpha$ and $V\vert_{U_{\alpha}} \rightarrow \alpha$ as $y\rightarrow +\infty$.  Set $r:=\varrho((x,y),(0,1))$. Since $\cosh(r(x,y))= \left( 1 + \frac{(y - 1)^2 + |x|^2}{2 y} \right)$, we get that $r(x,y)\rightarrow +\infty$ as $y\rightarrow +\infty$ and the corresponding claim of Theorem \ref{thphardyMaz} follows.\par
On the other hand, for any $\beta>0$, take $W_{\beta}:=\{(x_1,0,...,0,\beta) \in \mathbb{R}^{N}_{+} \}$. Then, for any $\beta>0$, one has $V\vert_{W_{\beta}}\rightarrow 0$ as $x_1\rightarrow +\infty$. Furthermore, $r\vert_{W_{\beta}}\rightarrow +\infty$ if and only if $x_1 \rightarrow +\infty$ and $V\vert_{W_{\beta}}\sim \sqrt{\frac{\beta}{2}}\,e^{-r/2}$ as $r\rightarrow +\infty$.

\section{Proof of Theorem~\ref{mainphardy} and Corollary \ref{upl}}

Before proving Theorem \ref{mainphardy}, we recall some known results related to the symmetrization on the hyperbolic space. For any $\Omega \subset \hn$ and $x_0\in \hn$ fixed, denote with $\Omega^*$ the geodesic ball $B(x_0,r)$ having the same measure of $\Omega$. For $u \in C_{c}^{\infty}(\Omega),$ the hyperbolic symmetrization of $u$ is the unique nonnegative and decreasing function $u^*$ defined in $\Omega^*$ such that the level sets $\{ x\in \Omega^*:u^*(x)>t\}$ are concentric balls having the same measure of the level sets $\{ x\in \Omega:|u(x)|>t\}$.   See \cite{ALB} form more details.

\begin{lem}\label{polya-szego}
Let $p \geq 1$ and $ N \geq 2.$ For every $u,v \in C_{c}^{\infty}(\hn),$ there holds

\[
\int_{\hn} |\nabla_{\hn} u|^{p} \ {\rm d}v_{\hn} \geq \int_{\hn} |\nabla_{\hn} u^{*}|^{p} \ {\rm d}v_{\hn},
\]
\[
\int_{\hn} |u|^{p} \ {\rm d}v_{\hn} = \int_{\hn}  |u^{*}|^{p} \ {\rm d}v_{\hn},
\]
and
\[
\int_{\hn} |uv| \ {\rm d}v_{\hn} \leq \int_{\hn} u^{*} v^{*} \ {\rm d}v_{\hn},
\]
where $*$ denotes the hyperbolic symmetrization.
\end{lem}

Next we state a $p-$convexity lemma. The proof of the following lemma can be obtained as an application of Taylor's formula, we refer to \cite{gaz}  for further details.

\begin{lem}\label{convexity}
Let $p \geq 1$ and $\xi, \eta$ be real numbers such that $\xi \geq 0$ and $\xi - \eta \geq 0.$ Then

$$ (\xi - \eta)^{p} + p \xi^{p-1} \eta - \xi^{p} \geq
\left \{\begin{array}{ll}
\emph{max} \{ (p-1) \eta^2 \xi^{p-2}, |\eta|^{p} \},    & \text{if $p \geq 2$}\,,\\
  \frac{1}{2} p(p-1) \frac{\eta^2}{(\xi + |\eta|)^{2-p}}, & \text{if $1 \leq p \leq 2.$}
 \end{array}\right.$$
 \end{lem}

Now we turn to prove an \emph{optimal} inequality which is one of the key ingredient in proving Theorem~\ref{mainphardy}.
\begin{lem}\label{lemhardy}
For all  $v \in W^{1,p}(0, \infty)$ and $1< l \leq p,$ there holds

\begin{equation}\label{hardytype}
\int_{0}^{\infty} |v(r)|^{p-l} (\coth r)^{p - l} |v^{\prime}(r)|^l \ {\rm d}r \geq \left( \frac{p-1}{p} \right)^l \int_{0}^{\infty} \frac{|v(r)|^{p}}{r^p} \ {\rm d}r.
\end{equation}
Furthermore, the constant $\left( \frac{p-1}{p} \right)^{l}$ in \eqref{hardytype} is sharp.
%
\end{lem}

\begin{proof} We first prove the claim for $v \in C_{c}^{\infty}(0, \infty)$. 
Write
\begin{align*}
\int_{0}^{\infty} \frac{|v(r)|^{p}}{r^{p}} \ {\rm d}r & = \frac{-1}{p-1} \int_{0}^{\infty} |v(r)|^{p} \frac{d}{dr}(r^{-(p-1)}) \ {\rm d}r \\
& = \left( \frac{p}{p-1} \right) \int_{0}^{\infty} \frac{|v(r)|^{p-2} v(r) v^{\prime}(r)}{r^{p-1}} \ {\rm d}r
 \\
 &
 \leq  \left( \frac{p}{p-1} \right) \int_{0}^{\infty}  \frac{|v(r)|^{p-1} |v^{\prime}(r)|}{r^{p-1}} \ {\rm d}r \\
 & =  \left( \frac{p}{p-1} \right) \int_{0}^{\infty} \frac{|v(r)|^{\frac{{p(l-1)}}{l}}}{r^{\frac{p(l-1)}{l}}} \frac{|v(r)|^{\frac{{p-l}}{l}}|v^{\prime}(r)|}{r^{\frac{p-l}{l}}} \ {\rm d}r     \\
 & \leq  \left( \frac{p}{p-1} \right) \left( \int_{0}^{\infty} \frac{|v(r)|^p}{r^p} \ {\rm d}r \right)^{\frac{l-1}{l}} \left( \frac{|v(r)|^{p-l} |v^{\prime}(r)|^l}{r^{p-l}} \ {\rm d}r \right)^{\frac{1}{l}}\,.
\end{align*}
Since $\coth r \geq \frac{1}{r}$ for all $r>0$, we conclude
\[
\int_{0}^{\infty} |v(r)|^{p-l} (\coth r)^{p - l} |v^{\prime}(r)|^l \ {\rm d}r \geq \left( \frac{p-1}{p} \right)^l \int_{0}^{\infty} \frac{|v(r)|^{p}}{r^p} \ {\rm d}r.
\]
Now, noticing that by using Young inequality and the classical Hardy inequality with exponent $p$, we have
$$\int_{0}^{\infty}|v(r)|^{p}{\rm d}r +\int_{0}^{\infty}|v^{\prime}(r)|^{p}{\rm d}r \ge
c \int_{0}^{\infty} |v(r)|^{p-l} (\coth r)^{p - l} |v^{\prime}(r)|^l \ {\rm d}r,$$
the claim follows by density argument.

Next we turn to the optimality issue. For $\varepsilon > 0$ and $\delta > 0,$ consider

 $$ V_{\varepsilon}^{\delta}(r) :=
\left \{\begin{array}{ll}
r^{\frac{p-1 + \delta}{p}}, \quad \quad \quad  0 < r < \varepsilon \\
\varepsilon^{\frac{p-1+\delta}{p}}, \quad \quad \quad \varepsilon \leq r < 1 \\
\varepsilon^{\frac{p-1+\delta}{p}}(2- r), \ 1 \leq r < 2\\
0, \quad \quad \quad \quad \quad \quad r \geq 2.
 \end{array}\right.$$
 Clearly, $ V_{\varepsilon}^{\delta}(r) \in W^{1, p}(0, \infty)$ for $\varepsilon >0, \delta > 0.$ Furthermore, we have

 \begin{equation*}\label{rhs}
\int_{0}^{\infty} \frac{|V_{\varepsilon}^{\delta}(r)|^{p}}{r^p} \ {\rm d}r \geq \int_{0}^{\varepsilon} \frac{r^{p -1 + \delta}}{r^p} \ {\rm d}r  = \int_{0}^{\varepsilon} r^{\delta - 1} \ {\rm d}r.
 \end{equation*}
On the other hand, using the fact $\sinh r \geq r,$ we obtain

\begin{equation*}
 \begin{aligned}\label{lhs}
 &\int_{0}^{\infty} |V_{\varepsilon}^{\delta}(r)|^{p-l} (\coth r)^{p - l}  |(V_{\varepsilon}^{\delta}(r))^{\prime}|^l \ {\rm d}r  =\\
  & \left( \frac{p-1+\delta}{p} \right)^{l} \int_{0}^{\varepsilon} r^{\frac{(p-1+ \delta)(p-l)}{p}} (\coth r)^{p-l} r^{\frac{(\delta - 1)l}{p}} \ {\rm d}r \\
   & + \varepsilon^{p-1+ \delta} \int_{1}^{2} (2-r)^{p - l} (\coth r)^{p - l} \ {\rm d}r  \\
    & = \left( \frac{p-1+\delta}{p} \right)^{l} \int_{0}^{\varepsilon} r^{p - 1 + \delta - l} (\coth r)^{p - l} \ {\rm d}r + c \varepsilon^{ p-1+ \delta}  \\
   & \leq \left( \frac{p-1+\delta}{p} \right)^{l} (\cosh\varepsilon)^{p -l} \int_{0}^{\varepsilon} \frac{r^{p-1+\delta-l}}{(\sinh r)^{p -l}}  \ {\rm d}r+ c \varepsilon^{p-1+\delta}  \\
   &   \leq \left( \frac{p-1+\delta}{p} \right)^{l} (\cosh\varepsilon)^{p -l} \int_{0}^{\varepsilon} r^{\delta - 1} \ {\rm d}r +  c \varepsilon^{p-1+\delta}.
 \end{aligned}
 \end{equation*}
 Hence,

 \begin{equation*}
Q:= \inf_{v \in W^{1, p}(0, \infty) \setminus \{ 0\}} \dfrac{\int_{0}^{\infty} |v(r)|^{p-l} (\coth r)^{p - l} |v^{\prime}(r)|^l \ {\rm d}r}{\int_{0}^{\infty} \frac{|v(r)|^{p}}{r^p} \ {\rm d}r}
 \leq \left( \frac{p-1+\delta}{p} \right)^{l} (\cosh \varepsilon)^{p - l} +  c \delta \varepsilon^{p-1}.
 \end{equation*}
First letting $\varepsilon \rightarrow 0$, and then with  $\delta \rightarrow 0, $ we conclude that
\[
 Q \leq \left( \frac{p-1}{p} \right)^{l} \,.
 \]
This proves the optimality and concludes the proof.

 \end{proof}

\noindent{\bf Proof of Theorem~\ref{mainphardy} and of Corollary \ref{upl}}
\medskip

 By hyperbolic symmetrization, i.e., in view of Lemma~\ref{polya-szego},  we may assume $u \in C_{c}^{\infty}(\hn)$ nonnegative,
radially symmetric and non increasing. Hence,  to prove \eqref{eqphardy}, it is enough to show the validity of the following inequality

 \begin{align}\label{onephardy}
   &\int_{0}^{\infty} |u^{\prime}(r)|^{p} (\sinh r)^{N-1} \ {\rm d}r  -  \left(\frac{N-1}{p} \right)^{p}  \int_{0}^{\infty}  (u(r))^{p} (\sinh r)^{N-1}  \ {\rm d}r  \notag \\
 &  \geq
    (p-1) \left( \frac{N-1}{p} \right)^{p-2} \left( \frac{p-1}{p} \right)^2 \int_{0}^{\infty} \frac{(u(r))^p}{r^{p}} (\sinh r)^{N-1} \ {\rm d}r\,.
 \end{align}
\\

Let us define a  suitable transformation which allows to put the Poincar\'e term into evidence:

\[
v(r) := (\sinh r)^{\frac{N-1}{p}} u(r)
\]
so that

\[
v^{\prime}(r) = (u^{\prime}(r)) (\sinh r)^{\frac{N-1}{p}} + \left( \frac{N-1}{p} (\sinh r)^{\frac{N-1}{p}} \coth r \right) u,
\]
hence $v\in W^{1,p}(0,\infty)$, and 
\[
(u^{\prime}(r)) (\sinh r)^{\frac{N-1}{p}} = v^{\prime}(r) -  \left( \frac{N-1}{p} (\sinh r)^{\frac{N-1}{p}} \coth r \right) u.
\]
At this point we apply the $p$-convexity Lemma~\ref{convexity}. By taking $$\xi = \left( \frac{N-1}{p} \right) (\sinh r)^{\frac{N-1}{p}} \coth r u > 0 \quad \text{ and } \quad \eta = v^{\prime}(r)$$ and using Lemma~\ref{convexity} for $p \geq 2$, we obtain

\begin{align*}
|u^{\prime}(r)|^{p} (\sinh r)^{N-1} & \geq   (p-1) \left( \frac{N-1}{p} \right)^{p-2} v^{p-2}(r) (\coth r)^{p-2} (v^{\prime}(r))^2 \\
&+ \left( \frac{N-1}{p} \right)^{p} (\sinh r)^{N-1} (\coth r)^{p} u^{p}(r) \\
& - p \left( \frac{N-1}{p} \right)^{p-1} (\sinh r)^{\frac{(N-1)(p-1)}{p}} (\coth r)^{p-1} u^{p-1}(r) v^{\prime}(r) \\
& = (p-1) \left( \frac{N-1}{p} \right)^{p-2} v^{p-2}(r) (\coth r)^{p-2} (v^{\prime}(r))^2 \\ &+ \left( \frac{N-1}{p} \right)^{p} (\sinh r)^{N-1} (\coth r)^{p} u^{p}(r) \\
& - p \left( \frac{N-1}{p} \right)^{p-1} (\coth r)^{p-1} v^{p-1}(r) v^{\prime}(r).
\end{align*}
Integrating both sides of above inequality and applying Lemma~\ref{lemhardy} with $l = 2$, we get

\begin{align*}
\int_{0}^{\infty} |u^{\prime}(r)|^{p} (\sinh r)^{N-1} \ {\rm d}r & \geq  (p-1) \left( \frac{N-1}{p} \right)^{p-2}   \int_{0}^{\infty} v^{p-2}(r) (\coth r)^{p-2} (v^{\prime}(r))^2 \ {\rm d}r \\
& + \left( \frac{N-1}{p} \right)^{p} \int_{0}^{\infty}  (\coth r)^{p} v^{p}(r) \ {\rm d}r\\
 & -  \left( \frac{N-1}{p} \right)^{p-1} \int_{0}^{\infty}  (\coth r)^{p-1} \frac{d}{d r}(v(r))^{p} \ {\rm d}r \\
& \geq  (p-1) \left( \frac{N-1}{p} \right)^{p-2}  \left( \frac{p-1}{p} \right)^{2} \int_{0}^{\infty} \frac{v^{p}(r)}{r^{p}} \ {\rm d}r \\ & + \left(\frac{N-1}{p} \right)^{p} \int_{0}^{\infty} F(r) (v(r))^{p} \ {\rm d}r,
\end{align*}
where $F(r) := (\coth r)^{p} - \frac{p(p-1)}{N-1} \frac{(\coth r)^{p}}{\cosh^2 r}$
and in the integration by parts we have used the definition of $v$ and the fact that $N>p$.
Then, \eqref{onephardy} follows by showing that $F(r) \geq 1$ for all $r>0$ or equivalently that
\[
\tilde F(r): = (N-1) \cosh^{p} r - (N-1) \sinh^{p} r -  p(p-1) \cosh^{p-2} r \geq 0,
\]
for all $r>0$. By rewriting

\[
\tilde F(r)=\cosh^{p-2}r(N-1-p(p-1))+(N-1) \sinh^{2} r(\cos^{p-2} r- \sinh^{p-2} r) \,,
\]

we immediately infer that $\tilde F(r)$ is nonnegative provided that $N \geq 1 + p(p-1),$ and also the condition is necessary.  This completes the proof of Theorem~\ref{mainphardy}. \qed
\medskip

\noindent \bf Proof of Corollary \ref{upl}\rm. It suffices to notice that, by H\"older inequality:
\[\begin{aligned}
&\int_{{\mathbb H}^N}|u|^{p}\ {\rm d}v_{\hn}=\int_{{\mathbb H}^N}\frac{|u|}r\,|u|^{p-1}r\ {\rm d}v_{\hn}\\
&\le\left(\int_{{\mathbb H}^N}\frac{|u|^p}{r^p}\ {\rm d}v_{\hn}\right)^{\frac 1p}\left(\int_{{\mathbb H}^N}|u|^pr^{p'}\ {\rm d}v_{\hn}\right)^{\frac1{p'}}.
\end{aligned}\]
The conclusion follows by using inequality \eqref{eqphardy}.\qed

\section{Proof of Theorem \ref{superphardy}}\label{figSection}

Before proving Theorem \ref{superphardy} we collect here below the main properties of the weight $H_p$. This will clarify also the meaning of inequality {\eqref{eqphardy2}, see also Figure 1.

\begin{lem}\label{asymptotic}
Let $ H_{p} : \mathbb{R^{+}} \rightarrow \mathbb{R}$ be defined as in the statement of Theorem \ref{superphardy} with $p >  2$ and  $N \geq 1 + p(p -1)$. Then, the following holds \\

\begin{itemize}

\item[(a)] For all $r > 0,$ $ H_{p}(r) > 0$, $ H_{p}(r) \sim \left(\frac{N-p}{N-1}\right)^{p-2} \,\frac{1}{r^{p-2}}$ as $r \rightarrow 0^+$, and  $ H_{p}(r) \rightarrow 1^{-}$ as $r \rightarrow \infty.$ \medskip

\item[(b)] There exists a unique $r_{p} \in (0, \infty)$ such that $ H_{p}(r) \geq 1 $ for $r\in(0, r_{p}]$ and $H_{p}(r) < 1$ for $r\in(r_{p}, \infty).$

\end{itemize}
\end{lem}

\begin{proof}
We set \[
\tilde H_{p}(r) : =   \coth r - \left(\frac{p-1}{N-1}\right)  \frac{1}{r} \,, \quad r>0\,.
\]
Then, the property of $H_p$ can be readily deduced from that of $\tilde H_{p}$.

The sign and the asymptotics of $\tilde H_{p}$ follows from fact that
$$\coth r > \frac{1}{r}\  \mbox{in} \  (0, \infty)\,, \quad \coth r \sim \frac{1}{r}\ \mbox{as} \  r \rightarrow 0^+\,,  \quad  \mbox{and} \  \coth r \rightarrow 1 \ \mbox{as} \  r \rightarrow \infty.$$

\medskip

To prove assertion (b), we note that

 \begin{equation}\label{hdef}
\tilde H'_{p}(r) =(N-1)^{-1}\left( \frac{-(N-1)r^2+(p-1)\sinh^2 r}{r^2\sinh^2 r} \right)=: \frac{(N-1)^{-1}}{r^2\sinh^2 r}\,h(r) \,.
 \end{equation}
Since $h'''(r)=8(p-1)\cosh r \sinh r>0$ for all $r>0$, $h''(0)=-2(N-p)$, and $h'(0)=h(0)=0$ one readily deduces the existence of a unique $r_{0}>0$ such that $h(r) < 0$ in $(0, r_{0})$, $h(r_0)=0$
and $h(r) > 0$ in $( r_{0}, \infty).$ Hence, $ \tilde H'_{p}(r) < 0$ in $(0, r_{0})$  and $\tilde H'_{p}(r) > 0$ in $( r_{0}, \infty).$ This fact and assertion (a)  gives the existence of a unique $r_{p}\in (0,r_0)$ for which (b) holds where $r_{p}$ clearly satisfies
 \begin{equation}\label{rp}
\coth r_{p} - 1 - \frac{p-1}{N-1} \frac{1}{r_{p}}=0.
 \end{equation}

\end{proof}

\begin{center}
\begin{figure}[ht]
 \includegraphics[width=0.7\textwidth]{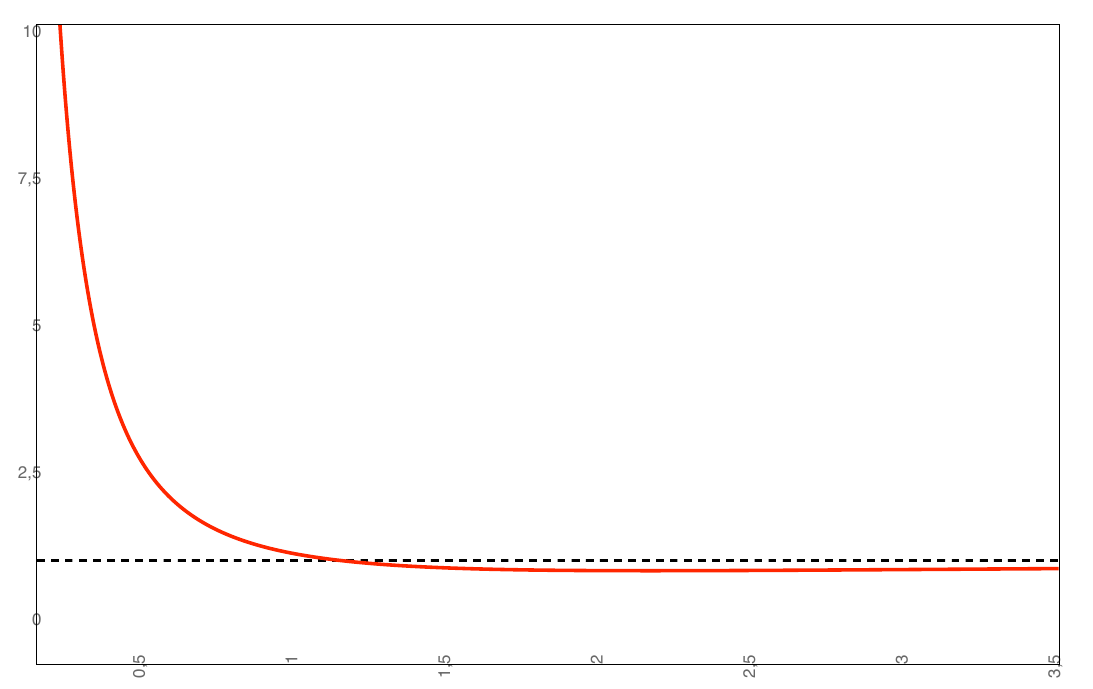}
\caption{The plot of $y=H_{p}(r) $ for $p=4$ and $N=13$. The dotted line is  $y=1$ and the intersection point of the two curves is the point $r_p$ as defined in Lemma \ref{asymptotic}-$(b)$.}
\end{figure}
\end{center}

\par
\bigskip\par

 \noindent {\bf{Proof of Theorem \ref{superphardy}}} \\
 The $p$-Laplacian operator in radial coordinates on the hyperbolic space writes

 \begin{equation}\begin{aligned}\label{radiallaplacian}
\Delta_{p, \hn} u(r)  :=  \Delta_{p}u(r) &= (p -1) | u^{\prime}(r)|^{p-2}  {u}^{\prime \prime}(r)  + (N-1) \coth r |u^{\prime}(r)|^{p-2} u^{\prime}(r)\\ &:=| u^{\prime}(r)|^{p-2}L_p u(r),
 \end{aligned}\end{equation}
where $L_{p} u(r) = (p-1) u^{\prime \prime}(r) + (N-1) \coth r u^{\prime}(r).$ \\
Set $g(r) = \left( \frac{r}{\sinh r} \right)^{\frac{(N-1)}{p}}$ and $f(r) = r^{\frac{p-N}{p}}$, some straightforward computations give

\begin{equation}
\begin{aligned}
 \label{comp1}
L_p g(r)& =  \frac{-(N-1)}{p} \left[ \frac{(N-1) - p(p-1)}{p} \frac{1}{\sinh^2 r} + \left(\frac{N-1}{p} \right) \right. \\
 & \left. + \frac{(p-1)(p - (N-1))}{p} \frac{1}{r^2} + \frac{(N-1)(p-2)}{p}  \frac{\coth r}{r}  \right] g(r) 
\end{aligned}
\end{equation}
 and

 \begin{align}\label{comp2}
 L_p f(r)= \left[ \frac{N(N-p)(p-1)}{p^2} \frac{1}{r^2} - (N-1)\coth r \frac{N - p}{p} \frac{1}{r} \right] f(r)
 \end{align}

Using \eqref{comp1} and \eqref{comp2},  we deduce for $\tilde g(r) = g(r) f(r),$

\begin{equation}\begin{aligned}\label{compute1}
L_p \tilde g (r)& = (L_pg(r)) f(r)+ ( L_p f(r) ) g(r) \\
 & + 2(p-1) \left( \frac{-(N-1)}{p} \coth r + \frac{N-1}{p} \frac{1}{r} \right) g(r) f^{\prime}(r) \\
 & = - \left[ \left( \frac{N-1}{p} \right)^{2} \tilde g + \frac{(p-1)^2}{p^2} \frac{1}{r^2} \tilde g + \frac{(p-1)(p-2)(N-1)}{p^2} \left( \frac{\coth r}{r} \right) \tilde g \right.  \\
 & \left. + \frac{(N-1)(N - 1 - p(p-1))}{p^2} \frac{1}{\sinh^2 r} \tilde g \right].
 \end{aligned}
\end{equation}

  In view of Eq. \eqref{radiallaplacian} and Eq. \eqref{compute1} we obtain

\begin{equation}\begin{aligned}\label{intermediate-step}
 -\Delta_{p} \tilde g - \left( \frac{N-1}{p} \right)^2 |\tilde g^{\prime}|^{p-2} \tilde g  =   \\
  \frac{(p-1)^2}{p^2} \frac{1}{r^2} |\tilde g^{\prime}|^{p-2} \tilde g +
 \frac{(p-1)(p-2)(N-1)}{p^2} \left( \frac{\coth r}{r} \right) |\tilde g^{\prime}|^{p-2} \tilde g   \\
  + \frac{(N-1)(N - 1 - p(p-1))}{p^2} \frac{1}{\sinh^2 r}  |\tilde g^{\prime}|^{p-2}\tilde g.
 \end{aligned}
\end{equation}

  Furthermore, we have
\begin{equation}\begin{aligned}\label{compute2}
\tilde g^{\prime}(r) & = (g^{\prime}(r))f(r) + (f^{\prime}(r))g(r)  \\
& = - \frac{1}{p} \left( (N-1) \coth r  - (p-1) \frac{1}{r} \right) \tilde g(r)\,.
 \end{aligned}
\end{equation}
Namely,
 \[
 |\tilde g^{\prime}(r)|^{p-2} = \left( \frac{N-1}{p}\right)^{p-2} H_{p}(r) {\tilde g}^{p-2}(r)\,,
 \]
with $H_{p}(r)$ as defined in the statement of Theorem \ref{hardyball}. On the other hand, a further computation using \eqref{compute2} and the fact $\coth r > \frac{1}{r},$ gives

\begin{equation}\begin{aligned}\label{compute3}
  |\tilde g^{\prime}(r)|^{p-2} & =  \frac{(p-1)^{p-2}}{p^{p-2} r^{p-2}} \left( \frac{N-1}{p-1} r \coth r  -  1 \right)^{p-2} \tilde g^{p-2}(r)  \\
  & \geq  \frac{(p-1)^{p-2}}{p^{p-2}}  \frac{\tilde g^{p-2}(r)}{r^{p-2}}.
 \end{aligned}
\end{equation}
 Substituting \eqref{compute3} in \eqref{intermediate-step} we conclude

\begin{equation*}\begin{aligned}
 -\Delta_{p} \tilde g - \left( \frac{N-1}{p} \right)^p H_{p}(r) {\tilde g}^{p-1}  \geq \frac{(p-1)^p}{p^p} \frac{1}{r^p}  {\tilde g}^{p-1}\\  +
 \frac{(p-1)^{p-1}(p-2)(N-1)}{p^p} \left( \frac{\coth r}{r} \right) \frac{1}{r^{p-2}} {\tilde g}^{p-1}  \\
  + \frac{(N-1)(N - 1 - p(p-1))}{p^2} \frac{1}{\sinh^2 r}  {\tilde g}^{p-1}  \\
  \geq \frac{(p-1)^{p-1} (N(p -2) + 1)}{p^p} \frac{1}{r^p}  {\tilde g}^{p-1} \\  + \frac{(N-1)(N - 1 - p(p-1))}{p^2} \frac{1}{\sinh^2 r}  {\tilde g}^{p-1}.
\end{aligned}
\end{equation*}
 This proves that  $\tilde g(r) = \left( \frac{r}{\sinh r} \right)^{\frac{N-1}{p}} r^{\frac{p-N}{p}} $ is a super-solution of the equation corresponding to \eqref{eqphardy2}. Hence, by Allegretto-Piepenbrink theorem for $p$-Laplacian setting, (for detail see  \cite[Theorem~2.3]{PT1})   \normalcolor
  inequality  \eqref{eqphardy2} follows immediately for functions in  $C_{c}^{\infty}(\hn\setminus\{ x_{0}\})$. To extend the inequality for functions belonging to
$C_{c}^{\infty}(\hn)$ one argues as in the proof of Proposition \ref{weight1}. Namely,
since $N>p$, the set $\{ x_{0}\}$ is compact and has  zero $p$-capacity, therefore
the completion of $C_{c}^{\infty}(\hn)$ and $C_{c}^{\infty}(\hn\setminus\{ x_{0}\})$ with respect to the norm$\left(\int_{\hn} |\nabla_{\hn} u|^p\,{\rm d}v_{\hn}\right)^{1/p}$
coincides (see \cite[Proposition A.1]{Dambrosio}). This concludes the proof. 
  \medskip

As a consequence of  Theorem \ref{superphardy} we have the following

 \begin{thm}\label{hardyball}
 Let $p \geq  2$ and  $N \geq 1 + p(p -1).$ Let $\Lambda_{p}$ be as in \eqref{LAMBDAp} and $r:=\varrho(x,x_0)$ with $x_0\in{\mathbb H}^N$ fixed. Then for $u \in C_{c}^{\infty}(B(x_{0}, r_{p}))$ there holds

\begin{equation}\begin{aligned}\label{eqphardy1}
\int_{B(x_{0}, r_{p})}&|\nabla_{\hn} u|^p\,{\rm d}v_{\hn}- \Lambda_p \int_{B(x_{0}, r_{p})}  |u|^{p} \ {\rm d}v_{\hn}\\  &\geq
 \frac{(p-1)^{p-1} (N(p -2) + 1)}{p^p}  \int_{B(x_{0}, r_{p})} \frac{|u|^p}{r^p} \ {\rm d}v_{\hn}   \\
 &+ \frac{(N-1)(N-1 -p(p-1))(p - 1)^{p-2}}{p^{p}} \int_{B(x_{0}, r_{p})} \frac{|u|^{p}}{\sinh^{p} r} \ {\rm d}v_{\hn}
\end{aligned}
\end{equation}
where $B(x_{0}, r_{p})$ is the geodesic ball of radius $r_{p}$ centered at $x_0$ and where we let, for $p>2$, $r_p=r_p(N)$ be the unique positive solution to the equation
 \[
\coth r_{p} - 1 - \frac{p-1}{N-1} \frac{1}{r_{p}}=0,
\]
whereas $r_2:=+\infty$ (namely $B(x_{0}, r_{2})=\hn$).

In particular, for every $p>2$ the map $N\mapsto r_p(N)$ is strictly increasing in $[1 + p(p -1),+\infty)$ and $ \lim_{N\rightarrow +\infty} r_p(N)=+\infty$ while, for every $N> 3$ the map $p\mapsto r_p$ is strictly decreasing in $(2,\frac{1+\sqrt{4N-3}}{2}]$.
 \end{thm}
 \begin{proof}

\noindent 
The proof readily follows by combining the statements of Theorem \ref{superphardy}  and Lemma \ref{asymptotic}. In particular equation \eqref{rp} implicitly defines a map $N\mapsto r_p(N)$.
 By differentiating in \eqref{rp} one gets
$$\frac{d}{dN} (r_p(N))=-\frac{(p-1)r_p \sinh^2r_p}{(N-1)h(r_p)}\,,$$
where the function $h$ is as defined in \eqref{hdef}. Since from the proof of Lemma \ref{asymptotic}-(b) we know that $h(r_p)<0$, we conclude that the map $N\mapsto r_p(N)$ is strictly increasing. On the other hand, equation \eqref{rp} also implicitly defines a map $p\mapsto r_p$.   In this case we get
$$\frac{d}{dp} (r_p)=\frac{r_p \sinh^2r_p}{(N-1)h(r_p)}<0\,.$$
Hence, the map $p\mapsto r_p(N)$ is strictly decreasing.
 \end{proof}

\par\bigskip\noindent
\textbf{Acknowledgments.} The first author is partially supported by the Research Project FIR (Futuro in Ricerca) no. RBFR13WJ6X
``Geometrical and qualitative aspects of PDE's'' (Italy). The second author is partially supported by
 the PRIN project no. 201274FYK7\textunderscore004 ``Aspetti Variazionali e Perturbativi nei Problemi Differenziali non Lineari''.
The third author is partially supported at the Technion by a fellowship of the Israel Council for Higher Education. The fourth author is partially supported by the PRIN project no. 2015HY8JCC ``Partial
Differential Equations and Related Analytic-Geometric Inequalities'' (Italy).
The first, second and fourth authors are members of the Gruppo Nazionale per l'Analisi Matematica, la Probabilit\`a e le loro Applicazioni (GNAMPA) of the Istituto Nazionale di Alta Matematica (INdAM).

\end{document}